\RequirePackage{lineno}
\documentclass[11pt,a4paper,reqno]{amsart}

\usepackage{color,graphicx}

\usepackage{amsmath,amssymb,amsthm}
\usepackage{a4wide,cite}

\def\eps{\epsilon}
\def\dt{\partial_t}

\def\curl{\text{\rm curl}}
\def\div{\text{\rm div}}

\def\E{\mathbf{E}}
\def\H{\mathbf{H}}
\def\S{\mathbf{S}}
\def\e{\mathbf{e}}
\def\h{\mathbf{h}}
\def\zero{\mathbf{0}}
\def\n{\mathbf{n}}
\def\g{\mathbf{g}}
\def\tnorm{|\!|\!|}

\def\Eup{\mathrm{E}}
\def\Hup{\mathrm{H}}

\def\O{{\Omega}}
\def\dO{{\partial\Omega}}

\def\ve{\mathbf{v}}
\def\vh{\mathbf{w}}

\def\V{\mathcal{V}}
\def\C{\mathcal{C}}

\def\TT{\mathbb{T}}
\def\PP{\mathbb{P}}
\def\NN{\mathbb{N}}
\def\RR{\mathbb{R}}
\def\K{K}
\def\dK{{\partial\K}}

\def\In{I^n}
\def\f{f}
\def\F{\mathcal F}
\def\Fih{{\F^{int}_h}}
\def\Fdh{{\F^\partial_h}}
\def\Oh{{\Omega_h}}

\def\ut{{t^{n-1}}}
\def\ot{{t^{n}}}

\newtheorem{theorem}{Theorem}

\newtheorem{method}{Method}

\theoremstyle{definition}
\newtheorem{remark}[theorem]{Remark}
\newtheorem*{remark*}{Remark}

\begin{document}
 \title[A space-time discontinuous Galerkin Trefftz method for Maxwell's equations]{A Space-Time Discontinuous Galerkin Trefftz Method for time dependent Maxwell's equations}

\author[H. Egger]{Herbert Egger}
\author[F. Kretzschmar]{Fritz Kretzschmar}
\author[S. M. Schnepp]{Sascha M. Schnepp}
 \author[T. Weiland]{Thomas Weiland}

\begin{abstract}
We consider the discretization of electromagnetic wave propagation problems
by a discontinuous Galerkin Method based on Trefftz polynomials.
This method fits into an abstract framework for space-time discontinuous 
Galerkin methods for which we can prove consistency, stability, and 
energy dissipation without the need to completely specify the approximation 
spaces in detail. Any method of such a general form results in an 
implicit time-stepping scheme with some basic stability properties.
For the local approximation on each space-time element, we then consider 
Trefftz polynomials, i.e., the subspace of polynomials that satisfy Maxwell's 
equations exactly on the respective element. 
We present an explicit construction of a basis for the local Trefftz spaces
in two and three dimensions and summarize some of their basic properties. 
Using local properties of the Trefftz polynomials, we can establish the 
well-posedness of the resulting discontinuous Galerkin Trefftz method.
Consistency, stability, and energy dissipation then follow immediately 
from the results about the abstract framework. 
The method proposed in this paper therefore shares many of the advantages of more 
standard discontinuous Galerkin methods, while at the same time, it yields a 
substantial reduction in the number of degrees of freedom and the cost for assembling. 
These benefits and the spectral convergence of the scheme are demonstrated in numerical tests.
\end{abstract}

\maketitle

 

\section{Introduction}

We consider the propagation of electromagnetic waves in an in-homogeneous 
isotropic dielectric linear medium governed by the time dependent Maxwell equations
%
 \begin{align}
\eps \dt \E - \curl \H &= \zero \qquad \text{in } \Omega\times (0,T) \label{eq:maxwell1},\\
\mu \dt \H + \curl \E  &= \zero \qquad \text{in } \Omega\times (0,T) \label{eq:maxwell2}.
\end{align}
As usual $\E$ and $\H$ denote the electric and magnetic field densities, $\eps$ is the electric permittivity, and $\mu$ the magnetic permeability; the material parameters are assumed to be piecewise constant and independent of time. At $t=0$, the fields are prescribed by initial conditions
\begin{align}
  \E(0) = \E^0, &\qquad \H(0) = \H^0  \qquad \text{on } \Omega.  \label{eq:maxwell3}
\end{align}
From a practical point of view it is reasonable to require that the computational domain is bounded, and we therefore assume that the fields additionally satisfy the boundary condition
\begin{align}
 \n \times \E  + \beta \n \times (\n \times \H) &= \n \times \g \qquad \qquad \qquad \text{on } \dO \times (0,T). \label{eq:maxwell4} 
\end{align}
A proper choice of the parameter $\beta \ge 0$ and the excitation $\g$ allows to model various physical situations, 
e.g., $\beta=0$ and $\g = 0$ leads to the condition for a perfect electric conductor, 
while setting $\beta = \sqrt{\eps/\mu}$ and $\g=0$ yields the first-order absorbing boundary condition. 

\medskip 
Problems of the form \eqref{eq:maxwell1} -- \eqref{eq:maxwell4} arise in various applications, 
e.g., in the analysis of wave guides and photonic crystals \cite{Tsukerman-PBG08,Pinheiro07}, in design of antennas \cite{BauFuLeuVa04,MunWei06}, or particle accelerators \cite{Gjonaj2006}. 
In all these applications the accurate and reliable simulation of the wave 
propagation is a key ingredient for the characterization, design, and optimization of corresponding electrical components. 
Since the seminal work of Yee \cite{Yee66}, finite difference time domain methods have been investigated, extended, and applied successfully. They can be considered state-of-the art for 
the numerical simulation of wave propagation in engineering  \cite{TafHag05,Wei96}. 
These methods are second order accurate on structured grids and utilize explicit time stepping which makes them 
easily parallelizable and very efficient in practice. Problems with non-aligned or curved boundaries and in-homogeneous materials require some non-trivial adaptions and local or implicit time stepping \cite{ZheChe01} has to be used to preserve stability in such cases which in turn may have a significant effect on the performance of the overall method.  
A more flexible framework for the spatial discretization is offered by finite element methods \cite{KuLaSch00,Monk03,Dem08}.
Their underlying variational framework allows to establish rigorous convergence results for rather 
general situations. Time dependent problems can then be treated by combination with appropriate time 
stepping schemes \cite{Jol03}. 
For scattering problems and more general applications involving unbounded domains, the boundary element 
method has been applied with great success, especially in the frequency domain \cite{HipSch02,BufHip03}, and more recently also for time dependent problems \cite{Mon15}. 
An even more flexible, but still variational, approach for constructing space discretizations for Maxwell's equations 
is offered by Discontinuous Galerkin methods \cite{PerSchMon02,PerSch03,HouPerSch04}. In principle, these methods allow to systematically couple different 
physical models and approximations even on hybrid and non-conforming meshes. Combined with explicit Runge-Kutta time 
stepping schemes, one can obtain efficient approximations for electromagnetic wave propagation problems of formally arbitrary order \cite{Fezoui2005,Cohen2006,AinMonMun06,WarHes08,GoeWarCle11,SchWei12}. 
To guarantee stability of the explicit time stepping methods, a somewhat restrictive CFL condition has to be satisfied. 
The treatment of locally adapted meshes therefore requires special techniques, like local or implicit time stepping,
in order to keep the computational cost acceptable.
Using a Galerkin approach not only for the discretization in space but also in time leads to 
space-time discontinuous Galerkin methods which have been investigated for the 
simulation of wave propagation problems only recently \cite{Monk14,Lilienthal14}. 
The discontinuous Galerkin framework allows to obtain approximations of formally arbitrary order on locally adapted meshes in space and time. The resulting methods typically lead to implicit time stepping schemes which are absolutely and unconditionally stable and slightly dissipative. 
A proper choice of approximation spaces and numerical fluxes even allows to obtain methods that are exactly energy preserving on the discrete level \cite{Lilienthal14}. 
%


In this paper, we consider such a space-time discontinuous Galerkin framework for the simulation of electromagnetic wave propagation problems. On each space-time element, we approximate the fields by \emph{Trefftz polynomials}, i.e., polynomial functions that satisfy Maxwell's equations exactly.
The idea to employ Trefftz functions for the numerical solution of partial differential equations is well established \cite{trefftz1926,Runge89,Jirousek97,Zielinski10}. 
%
One particular choice for the Trefftz functions consists in plane waves propagating in various directions. 
Corresponding discontinuous Galerkin methods have been proposed and analyzed for frequency domain problems in acoustics and electro-magnetics recently in \cite{Moiola:2011io,Hiptmair:2011p1917,Hiptmair2013,Badics14}. 
For time dependent wave propagation, one can construct complete sets of polynomial plane wave functions, and their use in combination with a discontinuous Galerkin framework has been proposed and investigated recently for acoustic problems \cite{Farhat14,Petersen:2009id}, and also electro-magnetic wave propagation problems \cite{Kretzschmar14,trefftz_abc}.
%
Like space-time discontinuous Galerkin schemes based on discretization with the full polynomial spaces \cite{Monk14,Farhat14}, the Trefftz method can be constructed for arbitrary approximation order and leads to implicit time-stepping schemes. 
We will establish the consistency and well-posedness of the single time-step problem, 
and also prove a basic energy dissipation relation. 
%
%
To illustrate the stability and convergence properties of the method, 
we report on numerical tests in which we observe spectral accuracy 
and optimal approximation orders. 
%
The decision to use only the Trefftz polynomials instead of the full polynomial approximation spaces yields a substantial reduction in the number of degrees of freedom on every element. 
In addition, the expensive computation of volume integrals can be completely avoided. 
This leads to a substantial reduction of the computational complexity while, at the same time, the flexibility, accuracy, and stability of the space-time discontinuous Galerkin framework is preserved.

\medskip

The outline of the paper is as follows: 
In Section~\ref{sec:stdg}, we introduce our notation and present an abstract space-time discontinuous Galerkin framework. We prove some general stability properties of a class of 
discretization schemes without specifying the approximation spaces in detail at this point.
In Section~\ref{sec:trefftz}, we present a systematic construction of a basis for the space of Trefftz polynomials, and we summarize some basic properties of these approximation spaces. 
In Section~\ref{sec:stdgt}, we present the discontinuous Galerkin Trefftz scheme, 
which results from employing piecewise polynomial Trefftz functions as approximation 
spaces in the abstract framework of Section~\ref{sec:stdg}. 
We formulate an equivalent but more convenient form of the method, 
establish its well-posedness, and derive the stability and energy dissipation
from the results of Section~\ref{sec:stdg}. 
In Section~\ref{sec:disc}, 
we collect several remarks about the properties of the discontinuous-Galerkin Trefftz method,
including a comparison with a corresponding method utilizing the full polynomial approximation spaces.
Section~\ref{sec:numerics} is devoted to numerical tests.
We observe spectral convergence of the scheme and optimal approximation 
order with respect to the spatial and temporal mesh size. 
In addition, we illustrate the stability of the method, and discuss 
the effect of numerical dissipation and dispersion. 
The paper concludes with a short summary.

\section{The space-time discontinuous-Galerkin framework} \label{sec:stdg}

The aim of this section is to introduce an abstract space-time discontinuous 
Galerkin framework for the discretization of electromagnetic wave propagation problems. 
By \emph{abstract} we mean that we do not specify the 
approximation spaces at this point. Even without explicit reference to the approximation spaces,
we can still prove some elementary properties that any method of this kind will inherit automatically.

\subsection{Notation} \label{sec:notation}

Let $\Omega \subset \RR^3$  be bounded polyhedral Lipschitz domain and 
$\Oh = \{\K\}$ be a non-overlapping partition of $\Omega$ into
simple regular elements $\K$, e.g., tetrahedral, parallel-epipeds, prisms, a.s.o.
We denote by $\Fih = \{\f = \partial \K \cap \partial \K', \ \K \ne \K' \in \Oh\}$ the set of element interfaces 
and by $\Fdh = \{\f = \partial \K \cap \partial\Omega, \ \K \in \Oh\}$ the set of faces on the boundary.
Let $\C(\Oh) = \{v : \Omega \to \RR : v|_{\K} \in \C(\K) \text{ for all } \K \in \Oh\}$ be the space of piecewise continuous functions. On element interfaces $\f= \partial\K_1 \cap \partial\K_2$, any piecewise smooth function $\E \in \C(\Oh)^3$ has formally two values 
$\E_1|_{\K_1} = \E_{\K_1}$ and $\E_2=\E|_{\K_2}$. We then denote by
\begin{align*}
\{\E\} = \frac{1}{2}(\E_1 + \E_2), 
\qquad 
[\n \times \E] = \n_1 \times \E_1 + \n_2 \times \E_2,
\end{align*}
the \emph{average} and the \emph{jump of the tangential component} of $\E$ on $f = \partial\K_1 \cap \partial\K_2$, respectively.
By $0=t_0 < t_1 < \ldots < t_N = T$ we generate a partition of $(0,T)$ 
into time intervals $\In = [\ut,\ot]$.

\subsection{A space-time discontinuous Galerkin framework} \label{sec:framework}

Let $\V_E^n,\V_H^n$, $n \ge 1$ be spaces of vector valued piecewise smooth functions 
over the partition $\Oh \times \In$ of the time slab $\Omega \times \In$. 
For the approximation of the initial boundary value problem \eqref{eq:maxwell1}--\eqref{eq:maxwell4}, we consider the following framework. 
\begin{method}[Abstract space-time discontinuous Galerkin method] \label{meth:stdg} $ $\\
Set $\E_h^0 = \E^0$, $\H_h^0 = \H^0$. For $n \ge 1$ find consecutively 
$(\E^n,\H^n) \in (\V^n_E, \V_H^n)$, 
such that 
\begin{align} \label{eq:stdg}
B^n(\E^n_h,\H^n_h;\ve,\vh) = R^n(\E^{n-1}_h,\H^{n-1}_h;\ve,\vh) %
\end{align}
for all $(\ve, \vh) \in (\V^n_E, \V^n_H)$ with bilinear forms $B^n$ and $R^n$ defined by
\begin{align*}
& B_n(\E,\H;\ve,\vh) = \\
&\qquad \quad \sum_{\K \in \Oh} \int_{\K \times \In}  (\eps \dt \E - \curl \H) \cdot\ve + (\mu \dt \H + \curl \E) \cdot\vh &&& \text{(volume terms)} \\
& \qquad + \sum_{\K \in \Oh} \int_{\K} \eps \E(\ut) \cdot \ve(\ut) + \mu \H (\ut) \cdot\vh(\ut) &&& \text{(temporal interface terms)} \\
& \qquad + \sum_{\K \in \Oh} \int_{\partial\K \times \In} \n \times (\H - \H^*) \cdot\ve - \n \times (\E - \E^*) \cdot\vh &&& \text{(spatial interface terms)} 
\end{align*}
\begin{align*}
&R^n( \E,  \H,\ve,\vh) = \\
&  \qquad + \sum_\K\int_{\K} \eps  \E(\ut) \cdot \ve(\ut) + \mu  \H(\ut) \cdot\vh(\ut) &\qquad&&  \text{(temporal interface terms)} \\
& \qquad -\int_{\partial \Omega \times \In} \n \times \g \cdot \vh. &&& \text{(boundary terms)} 
\end{align*}
On internal faces $\f \in \Fih$ between adjacent elements, we set
\begin{align*}
\E^* = \{\E\}, 
\qquad 
\H^* = \{\H\}, 
\end{align*}
and on the boundary faces $\f \in \Fdh$ we choose $\E^*=\beta \n \times (\H \times \n)$ and $\H^*=\H$, respectively.
\end{method}

\begin{remark}
We call Method~\ref{meth:stdg} \emph{abstract} since the space$(\V_E^n,\V_H^n)$ has not been defined yet and thus the method is not implementable at this stage.
Various generalization are possible: more general numerical fluxes, e.g., $\E^* = \{\E\} - \gamma [\n \times \H]$ and $\H^* = \{\H\} + \delta [\n \times \E]$, can be considered and different spatial meshes may be used on every time slab. Our arguments easily cover also such cases. For ease of presentation, we however stick with the simple setting stated above.
\end{remark}

\subsection{Basic properties of the abstract discontinuous Galerkin method} \label{sec:properties}

Without further specifying the approximation space $(\V_E^n,\V_H^n)$, 
we can already derive some basic properties that any space-time discontinuous Galerkin method of the 
above form will share. 

\begin{theorem} \label{thm:implicit}
Assume that the variational problem \eqref{eq:stdg} is solvable for every $n \ge 1$. 
Then Method~\ref{meth:stdg} ammounts to an implicit time stepping scheme. 
\end{theorem}
\begin{proof}
It follows directly from \eqref{eq:stdg} that $(\E_h^n,\H_h^n)$ depends only on $(\E_h^{n-1},\H_h^{n-1})$ 
and on the data $\g|_{\partial\Omega \times \In}$. The approximations $(\E_h^n,\H_h^n)$ can therefore be computed consecutively, provided that \eqref{eq:stdg} is uniquely solvable for every time slab $\Omega \times \In$. 
Since a linear problem has to be solved in every step, one obtains an implicit scheme.
\end{proof}

The following properties are the key building blocks for a convergence analysis of 
discontinuous Galerkin methods of the above form; see \cite{Monk14,Lilienthal14} for similar results in the context of slightly different methods. 
Our first observation is that Method~\ref{meth:stdg} 
yields a consistent approximation to the initial boundary value problem under consideration.

\begin{theorem}[Consistency] \label{thm:consistency} $ $\\
Let $(\E,\H)$ be a picewise smooth solution of problem \eqref{eq:maxwell1}--\eqref{eq:maxwell4}.
Then 
\begin{align} \label{eq:galorth}
B^n(\E,\H;\ve,\vh) = R^n(\E,\H;\ve,\vh),
\end{align}
for all piecewise smooth test functions $\ve$, $\vh$ and all $n \ge 1$. 
\end{theorem}
\noindent 
Any sufficiently smooth solution of problem \eqref{eq:maxwell1}--\eqref{eq:maxwell4} thus 
satisfies the discrete variational principle \eqref{eq:galorth} and we therefore call Method~\ref{meth:stdg} 
\emph{consistent} with the initial boundary value problem.
\begin{proof}
We divide the proof into several small steps:\\
(i) Since $(\E,\H)$ is a solution of Maxwell's equations, the volume terms in $B^n$ vanish. \\
(ii) By tangential continuity of the fields and the definition of the numerical fluxes, we obtain $\n \times \E = \n \times \E^*$ and $\n \times \H = \n \times \H^*$ on the element interfaces. Hence the interface terms in the definition of $B^n$ vanish.\\
(iii) By continuity of the fields in time, we get $\E^{n-1}(\ut) = \E^n(\ut)$ and $\H^{n-1}(\ut) = \H^n(\ut)$. 
Therefore, the temporal interface terms in $B^n$ and $R^n$ cancel out. \\
(iv) The remaining boundary terms in $B^n$ and $R^n$ cancel due to the boundary condition.
\end{proof}

As a next step, we prove a kind of coercivity statement, which is the basis for proving stability 
and well-posedness of the space-time discontinuous Galerkin Trefftz Method that will be presented 
in Section~\ref{sec:stdgt}.

\begin{theorem}[Stability] \label{thm:coercivity} $ $ \\
Let $n \ge 1$ and $(\V_E^n,\V_H^n)$ be a piecewise smooth approximation space. Then
\begin{align} \label{eq:stability}
B^n(\ve,\vh;\ve,\vh) \ge \tfrac{1}{2} \tnorm(\ve,\vh)\tnorm_{\Omega \times \In}^2 \quad \text{for all } (\ve,\vh) \in (\V_E^n,\V_H^n),
\end{align}
where 
\begin{align*}
\tnorm(\ve,\vh)\tnorm_{\Omega \times \In}^2 
&= \|\eps^{1/2} \ve(\ot)\|^2_{\Omega} + \|\mu^{1/2} \vh(\ot)\|^2_{\Omega} 
\\ & \qquad \qquad 
+ \|\eps^{1/2}\ve(\ut)\|^2_{\Omega} + \|\mu^{1/2} \vh(\ut)\|^2_{\Omega}  + 2\beta \|n \times \H\|_{\partial\Omega}^2. 
\end{align*}
\end{theorem}
\begin{proof}
The result follows by elementary manipulations. For convenience of the reader, we provide a detailed proof in the appendix.
\end{proof}

\begin{remark}
Since $\tnorm(\cdot,\cdot)\tnorm_{\Omega \times \In}$ is only semi-norm on $(\V_E^n,\V_H^n)$, 
the above stability estimate is not sufficient to guarantee the unique solvability of the variational problem \eqref{eq:stdg} for general approximation spaces. 
For piecewise tensor product polynomials, one can prove a different stability estimate, namely 
\begin{align} \label{eq:l2stability}
B^n(\ve,\vh;\ve,\vh) 
\ge  \frac{c}{\triangle \ot} \big( \|\eps^{1/2} \ve\|^2_{\Omega \times \In} + \|\mu^{1/2}\vh\|_{\Omega \times \In}^2\big)
- \tnorm(\ve,\vh)\tnorm_{\Omega \times \In}^2 
\end{align}
with some constant $c>0$ independent of the mesh size. 
A combination of \eqref{eq:l2stability} and \eqref{eq:stability} then 
allows to show that the bilinear form $B^n$ 
is stable in a space-time $L^2$-norm, and therefore Method~\ref{meth:stdg} is well-defined; 
let us refer to \cite{Lilienthal14,Monk14} for details. 
In this paper, we utilize a different choice of approximation spaces,
which allows us to establish well-definedness of the resulting space-time 
discontinuous Galerkin method \eqref{eq:stdg} in a more direct manner. 
\end{remark}

The final result of this section is concerned with conservation of energy 
which is a basic physical principle and which immediately 
implies uniform bounds for the energy.  
This is the starting point for proving existence of solutions on the continuous level. 
Let $\mathcal E(t) = \frac{1}{2} \int_\Omega \eps |\E(t)|^2 + \mu |\H(t)|^2$ 
denote the total electromagnetic energy contained in $\Omega$. 
Then any solution $(\E,\H)$ of Maxwell's equations \eqref{eq:maxwell1}--\eqref{eq:maxwell2} satisfies the energy identity
\begin{align*}
\frac{1}{2} \big( \|\eps^{1/2} \E\|^2_\Omega +  \|\mu^{1/2} \H\|_\Omega^2 \big) \Big|_{\ut}^{\ot}  
= - \int_{\partial\Omega \times \In} \n \times \E \cdot \H.
\end{align*}
This is a special instance of the Poynting theorem, which asserts that the change of the electromagnetic energy is due to energy flux $\S \cdot \n = (\E \times \H) \cdot \n = \n \times \E \cdot \H$ over the boundary. 
The boundary condition \eqref{eq:maxwell4} can be used to replace $\n \times \E$ on the right hand side.
A similar energy relation now also holds for every solution of Method~\ref{meth:stdg}, independent of the specific choice of the approximation spaces. 

\begin{theorem} \label{thm:energy}
Let 
$(\E_h^n,\H_h^n)$ be a solution of Method~\ref{meth:stdg}. Then 
\begin{align} \label{eq:energy}
& \frac{1}{2} \big(  \|\eps^{1/2} \E_h^n\|^2_\Omega + \|\mu^{1/2} \H_h^n\|_\Omega^2 \big) \Big|_{\ut}^{\ot} 
=   -\int_{\partial\Omega \times \In} \beta |\n \times \H_h^n|^2 + \n \times \g \cdot \H_h^n  \\
& \qquad \qquad 
- \tfrac{1}{2} \|\eps^{1/2} \big( \E_h^n(\ut) - \E_h^{n-1}(\ut)\big) \|^2_\Omega 
- \tfrac{1}{2} \|\mu^{1/2} \big(\H_h^{n}(\ut) - \H_h^{n-1}(\ut) \big)\|_\Omega^2. \notag
\end{align}
\end{theorem}
\begin{proof} 
The result follows from testing \eqref{eq:stdg} with $\ve=\E_h^n$ and $\vh=\H_h^n$ and elementary manipulations. 
For convenience of the reader, a detailed proof is given in the appendix. 
\end{proof}

\begin{remark}
The terms in the second line of the discrete energy relation \eqref{eq:energy} 
amount to some artificial numerical dissipation which is due to the implicit nature of the time discretization. 
This could be circumvented by employing continuous approximations in time \cite{Fezoui2005,Lilienthal14,Monk14}.
The jump terms can also be interpreted as a penalization of the discontinuities in time 
that provide some extra stability.
For higher approximation order, the amount of numerical dissipation is negligible and does not negatively affect the approximation order of the scheme.   
Numerical fluxes of the form 
$\E^* = \{\E\} - \gamma [\n \times \H]$ and $\H^* = \{\H\} + \delta [\n \times \E]$,
would give rise to additional dissipative terms of the form
$
\sum_{f \in \Fih} \gamma \| \n \times \H \|^2_{f \times \In} + \delta \|\n \times \E\|^2_{f \times \In}.
$
As has been shown in \cite{Petersen:2009id,Lilienthal14,Monk14}, these additional dissipative terms may have a positive influence 
on the convergence order with respect to the spatial mesh size. 
\end{remark}

\section{Trefftz polynomials} \label{sec:trefftz}

A standard choice for approximation spaces in space-time discontinuous Galerkin methods 
consists of piecewise tensor-product polynomials \cite{Monk14,Lilienthal14}. 
In this paper, we will utilize only the subspace of the full polynomial space consisting of 
those polynomials which satisfy Maxwell's equations exactly. 
In the following, we formally introduce these spaces of Trefftz polynomials, and we summarize some of their basic properties.

\subsection{Construction of Trefftz polynomials}

Let $\K \in \Oh$ be an element of the mesh and $\In = [\ut,\ot]$ be some time interval of our discretization.
We assume that the parameters $\eps$ and $\mu$ are constant on $\K \times \In$. 
Let $D \subset \RR^d$ and $\PP_p(D)$ be the space of all polynomials on $D$ of degree less or equal to $p$. We denote by 
\begin{align} \label{eq:trefttz}
\TT_p(\K \times \In) = \{(\E,\H) \in \PP_{p}(\K \times I)^6 : \eps \dt \E - \curl \H = 0, \ \mu \dt \H + \curl \E = 0\}
\end{align}
the space of \emph{Trefftz polynomials}, i.e., of vector valued polynomials satisfying Maxwell's equations \eqref{eq:maxwell1}--\eqref{eq:maxwell2} 
on $\K \times \In$. 
Note that any element of $\TT_p(\K \times \In)$ has six coupled electromagnetic field components.
The following characterization will be the starting point for a systematic construction of a basis for $\TT_p(\K \times \In)$.
\begin{theorem}[Characterization] \label{thm:rep} $ $ \\
For any $\widetilde \E,\widetilde \H \in \PP_p(\K)$ there is a unique $(\E,\H) \in \TT_{p}(\K \times \In)$ with $\E(\ut)=\widetilde \E$, $\H(\ut) = \widetilde \H$. 
\end{theorem}
\begin{proof}
Let $(\E,\H)$ be in $\TT_p(\K \times \In)$. Then $\E$ and $\H$ can be expanded as 
\begin{align*}
\E(x,t) = \sum_{m=0}^p \e_{m}(x) (t-\ut)^m, \qquad \H(x,t) = \sum_{m=0}^p \h_{m}(x) (t-\ut)^m
\end{align*}
with $\e_{m}$, $\h_{m} \in \PP_{p-m}^3(\K)$. 
From $\E(\ut) = \widetilde \E$ and $\H(\ut) = \widetilde \H$, we directly obtain
$$
\e_{0}(x)=\E(x,\ut)
\quad \text{and} \quad  
\h_0(x)=\H(x,\ut).
$$ 
Inserting the expansion into the Maxwell equations and comparing powers of $t$, we get 
\begin{align} \label{eq:trefftz}
m \eps \e_{m} = \curl \h_{m-1}
\qquad \text{and} \qquad 
m \mu \h_{m} = -\curl \e_{m-1}, 
\end{align}
for all $m=1,\ldots,k$. This allows us to compute $e_{m}$ and $\h_{m}$ recursively. 
Note that $e_{2l}$ and $h_{2l+1}$ depend only on $\bar \E$ while $e_{2l+1}$ and $h_{2l}$ depend only on $\bar \H$.
\end{proof}

\begin{remark}
The proof of Theorem~\ref{thm:rep} also provides a constructive way to efficiently generate a basis for the space of Trefftz polynomials. One only has to choose a Trefftz basis for the initial values $(\widetilde \E,\widetilde \H) \in \PP_p(\K)^6$ and then propagate the fields in time. 
The construction also reveals that the Trefftz polynomials have coupled electric and magnetic components, in general.
In particular, non-constant functions of the form $(\E,\zero)$ or $(\zero,\H)$ do not lie in $\TT_p(\K \times \In)$.
\end{remark}
As a direct consequence of the previous characterization, we obtain
\begin{theorem}
$\dim \TT_{p}(\K \times \In) = (p+3)(p+2)(p+1)$. 
\end{theorem}
\begin{proof}
The previous lemma shows that any $(\E,\H) \in \TT_{p}(\K \times \In$ can be represented uniquely by its initial values $(\E(\ut),\H(\ut))=(\widetilde \E,\widetilde \H) \in \PP_{p}(\K)^6$.
The assertion then follows by noting that $\dim \PP_{p}(\K) = (p+3)(p+2)(p+1)/6$ and counting the dimensions.
\end{proof}

The following stability estimate will allow us in the next section to to prove coercivity of the bilinear form $B^n$ and thus to ensure the well-posedness of the space-time discontinuous Galerkin method based on Trefftz polynomials.

\begin{theorem} \label{thm:norm}
Let $\K \subset \Oh$ be an element of the mesh and $\K \times \In$ denote the corresponding space-time element.
Then for all $(\E_h^n,\H_h^n) \in \TT_p(\K \times \In)$ there holds 
\begin{align*}
 \eps \|\E_h^n\|^2_{\K \times \In} + \mu \|\H_h^n\|^2_{\K \times \In} 
 &\le C(p,\K,\In) \; \big( \eps \|\E(\ut)\|^2_{\K} + \mu \|H(\ut)\|^2_{\K}\big)
\end{align*}
with a constant $C(p,\K,\In)$ only depending on the polynomial degree, the spatial element, the size of the time interval,
and the material parameters.
\end{theorem}
\begin{proof}
The usual energy argument yields 
\begin{align*}
&\frac{1}{2} \frac{d}{dt} \Big( \eps \|\E\|^2_K + \mu \|\H\|_K^2 \big) 
= (\curl \H, \E)_K - (\curl \E, \H)_K \\
&\qquad \qquad = (\n \times \H, \E)_{\dK}
 \le \|\H\|_\dK \|\E\|_{\dK} 
\le c \big( \eps \|\E\|^2_K + \mu \|\H\|_K^2 \big).  
\end{align*}
For the last estimate, we used a discrete trace inequality \cite{WarHes08} and Young's inequality. 
By employing Gronwall's lemma, we then obtain 
\begin{align*}
\eps \|\E(t)\|^2_K + \mu \|\H(t)\|_K^2 
&\le e^{c|t-\ut|} \big( \eps \|\E(\ut)\|^2_K + \mu \|\H(\ut)\|_K^2 \big),
\end{align*}
The assertion of the theorem now follows by integration with respect to the time variable. 
\end{proof}
\begin{remark}
It is possible to explicitly describe the dependence of $C(p,\K,\In)$ on the polynomial degree and on the spatial and temporal mesh size. For $\triangle t \le c p^2/h$, the constant can be shown to be bounded independent of the meshsize. 
\end{remark}

Before we proceed, let us also briefly discuss related constructions of approximation spaces,
which may be useful in practice and which will actually be used in our numerical experiments.

\subsection{Incorporation of divergence constraints}

Assume that a Trefftz function additionally satisfies
$\div \eps \E(\ut) = \div \eps \H(\ut)=0$. 
Then, by taking the divergence of \eqref{eq:maxwell1} and \eqref{eq:maxwell2}, 
we conclude that 
\begin{align} \label{eq:gauge}
\div \eps \E = 0, \qquad \div \mu \H = 0 \qquad  \text{on } \K \times \In.
\end{align}
These constraints, which express the absence of electric charges and magnetic monopoles,
can easily be incorporated in the construction of the local Trefftz polynomials.
Following the arguments of the construction in the previous section, we obtain
\begin{theorem}
Denote the space of divergence free local  Trefftz polynomials by
\begin{align*}
\widetilde \TT_p(\K \times \In) 
&= \{(\E,\H) \in \TT_p(\K \times \In): \div \eps \E = 0, \ \div \mu \H=0\}
\end{align*}
Then $\dim \widetilde \TT_{p}(\K \times \In) = \frac{1}{3} (p+1)(p+2)(2p+9)$. 
\end{theorem}
\begin{proof}
Note that $\div \PP_p(\K\times\In)^3 = \PP_{p-1}(\K\times\In)$ and that $\dim \PP_{p-1}(\K \times \In) = \tfrac{1}{6}p(p+1)(p+2)$.
The two constraint conditions thus yield $p(p+1)(p+2)/3$ additional constraints, 
and the result follows by counting of the dimensions.
\end{proof}

\begin{remark}
A systematic construction of a basis for $\widetilde \TT_p(\K \times \In)$ can be done as follows: 
(i) Choose a basis for $(\E(0),\H(0))$ spanning $\{(\e,\h) \in \PP_p(\K) : \div \e = 0, \ \div \h=0\}$. 
This can be achieved by taking curls of appropriate polynomials of order $p+1$. 
(ii) Extend this polynomial basis for the initial values $t=0$ to the space-time element $\K \times \In$ 
by solving Maxwell's equations, which can be achieved utilizing recurrence relations similar to those used in the proof of Theorem~\ref{thm:rep}. 
An explicit construction of a basis for $\TT_{p}(\K \times \In)$ consisting of \emph{polynomial plane waves} has been 
given in \cite{trefftz_abc}.
\end{remark}

\subsection{Lower dimensional approximations} \label{sec:lda}

Under symmetry assumptions, Maxwell's equations can be reduce to a simpler setting. 
For illustration and later reference, let us consider one such case in more detail. 
This will also be the setting for our numerical tests in Section~\ref{sec:numerics}.

Assume that the domain and the fields are homogeneous in the $z$-direction and that the 
electric field is polarized in this direction. 
The electromagnetic fields then have the form
$\H=\left(\Hup_1,\Hup_2,0\right)$ and $\E = \left(0,0,\Eup_3\right)$ with $\Hup_1$, $\Hup_2$, and $\Eup_3$ independent of $z$.
This setting is known as the TM mode in electrical engineering. 
We then define 
\begin{align*} 
\TT_p^{2D}\left(\K \times \In\right)&= \big{\{}\left(\E,\H\right)  \in \TT_p\left(\K \times \In\right) : \E=\left(0,0,\Eup_3\right), \ \H=\left(\Hup_1,\Hup_2,0\right)\big{\}}
\end{align*}
with components $\Hup_1,\Hup_2,\Eup_3$ independent of $z$. 
The superscript $^{2D}$ is used here to distinguish this setting from the general three-dimensional case.
Similar as before, we also consider the space $\widetilde \TT_p^{2D}(\K \times \In) = \{(\E,\H) \in \TT_p^{2D}(\K \times \In) : \ \div \H = 0\}$ of the corresponding divergence free Trefftz polynomials. 
Note that the divergence free condition on $\E$ is satisfied automatically. 
The construction of a basis for the polynomial Trefftz spaces can now be done similar to the general case, and we obtain
\begin{theorem}
$\dim \TT_p^{2D}(\K \times \In) = \frac{3}{2}(p+1)(p+2)$ and $\dim \widetilde \TT_p^{2D}(\K \times \In) = (p+1)(p+3)$. 
\end{theorem}
The proof of these assertions follows similar to that of Theorem~\ref{thm:rep} by counting arguments.
\begin{remark}
Since we assumed homogeneity of the domain in $z$-direction, 
we can express $\Omega = \Omega' \times Z$ with $\Omega' \subset \RR^2$ and $Z$ being some interval. It is then natural to consider a tensor product mesh with elements 
$\K = \K' \times Z$ where $\K' \subset \Omega'$ is an element of a partition of $\Omega'$.
For the actual implementation, it therefore suffices to consider spatial meshes in two dimensions.
\end{remark}

Assuming homogeneity in two coordinate directions would allow to reduce Maxwell's equations to 
a one-dimensional setting; see \cite{Kretzschmar14} for details about corresponding results.

\section{The space-time discontinuous Galerkin Trefftz method} \label{sec:stdgt}

We will now utilize the Trefftz polynomials for the local approximation in the space-time discontinuous Galerkin framework introduced in Section~\ref{sec:stdg}. 
We therefore choose
$$
(\V_E^n, \V_H^n) :=  \TT_p(\Oh \times \In) := \{ v \in L^2(\Omega \times \In) : v |_{\K \times \In} \in \TT_p(\K \times \In)\},
$$
i.e., we approximate the Fields by of piecewise Trefftz polynomials of order $p$. 
One might as well use one of the other polynomial Trefftz spaces introduced in the previous section.
The special properties of the Trefftz polynomials leads to some simplifications in the formulation of the space-time discontinuous Galerkin method, which we explain next.

\subsection{Space-time discontinuous Galerkin Trefftz method} 

Since piecewise Trefftz polynomials satisfy the Maxwell equations on every element, 
the volume terms in the definition of $B^n$ in Method~\ref{meth:stdg} drop out. 
By the usual rearrangement of the interface terms \cite{WarHes08}, we obtain 
\begin{align} \label{eq:identity}
&
\sum_{\K \in \Oh} \int_{\dK}  \n \times (\H - \H^*) \cdot \ve - \n \times (\E-\E^*) \cdot \vh \\
&
\qquad \qquad =  \sum_{\f \in \Fih} \int_\f [n \times \H] \cdot \{\ve\} - [\n \times \E] \cdot \{\vh\} 
 + \sum_{\f \in \Fdh} \int_\f  \beta (\n \times \H) (\n \times \vh) - \n \times \E \cdot \vh. \notag
\end{align}
A detailed derivation is given in the appendix. 
When using piecewise Trefftz polynomials as approximation spaces, the abstract space-time discontinuous Galerkin method of Section~\ref{sec:stdg} can therefore be rephrased equivalently as follows.

\begin{method}[Space-time discontinuous Galerkin Trefftz method] \label{meth:stdgt} $ $\\
Set $\E_h^0=\E^0$, $\H_h^0=\H^0$. For $n \ge 1$ find $(\E_h^n,\H_h^n) \in \TT_p(\Oh \times \In)$ such that 
\begin{align} \label{eq:stdgt}
B^n(\E_h^n,\H_h^n;\ve,\vh) = R^n(\E_h^{n-1},\H_h^{n-1};\ve,\vh) 
\end{align}
for all $(\ve,\vh) \in \TT_p(\Oh \times \In)$ with $B^n$ and $R^n$ defined by
\begin{align*}
&B^n(\E,\H;\ve,\vh) = \\
& \qquad \quad   \sum_{\f \in \Fih} \int_{\f \times \In}  [n \times \H] \cdot \{\ve\} - [\n \times \E] \cdot \{\vh\} &&& \text{(spatial interface terms)} \\
& \qquad + \sum_{\f \in \Fdh} \int_{\f \times \In} \beta (\n \times \H)\cdot (\n \times \vh) - \n \times \E \cdot \vh &&& \text{(boundary terms)} \\
& \qquad + \sum_{\K \in \Oh} \int_\K \eps \E(\ut) \cdot \ve(\ut) + \mu \H (\ut) \cdot\vh(\ut) &&& \text{(temporal interface terms)}
\end{align*}
\begin{align*}
&R^n( \E, \H,\ve,\vh) = \\
& \qquad  - \sum_{\f \in \Fdh} \int_{\f \times \In} \n \times \g \cdot \vh &&& \text{(boundary terms)} \\
&  \qquad + \sum_{\K \in \Oh} \int_\K \eps  \E(\ut) \cdot \ve(\ut) + \mu  \H(\ut) \cdot\vh(\ut) &&&  \text{(temporal interface terms)}
\end{align*}
\end{method}

Note that Method~\ref{meth:stdgt} is equivalent to Method~\ref{meth:stdg} with $(\V_E^n,\V_H^n) = \TT_p(\Oh \times \In)$. We can therefore use the same symbols for the forms $B^n$ and $R^n$. 
%

\subsection{Properties of the space-time discontinuous Galerkin Trefftz method}

It remains to show that Method~\ref{meth:stdgt} is well-defined, i.e., we have to 
verify that the discrete variational problems \eqref{eq:stdgt} are uniquely solvable.
This follows from the fact that the semi-norm $\tnorm(\ve,\vh)\tnorm_{\Oh \times \In}$
used in Theorem~\ref{thm:coercivity} is actually a norm on the space of piecewise Trefftz polynomials.
\begin{theorem}
For all piecewise Trefftz-polynomials $(\ve,\vh) \in \TT_p(\Oh \times \In\}$ there holds
$$
\tnorm(\ve,\vh)\tnorm_{\Oh \times \In} \ge \sum_{\K \in \Oh}  \frac{1}{C(p,\K,\In)} \Big( \eps \|\ve\|^2_{\K \times \In} + \mu \|\vh\|^2_{\K \times \In} \Big),
$$
with constant $C(p,\K,\In)$ taken from Theorem~\ref{thm:norm}. 
\end{theorem}
\begin{proof}
The estimate follows directly by omitting the spatial interface terms in $\tnorm(\cdot,\cdot) \tnorm_{\Omega \times \In}$, 
applying the estimate of Theorem~\ref{thm:norm} on every element, and summing over all elements.
\end{proof}

All remarks and assertions about the abstract space-time discontinuous Galerkin method of Section~\ref{sec:stdg} 
now carry over verbatim to the Trefftz method. 
For completeness, we summarize the basic properties 

\begin{theorem} \label{thm:stdgt}
Method~\ref{meth:stdgt} is a consistent and well-defined implicit time-stepping scheme
and the  approximations $(\E_h^n,\H_h^n) \in \TT_p(\Oh \times \In)$ obtained with Method~\ref{meth:stdgt} satisfy the discrete energy dissipation relation 
\begin{align*}
& \frac{1}{2} \big(  \|\eps^{1/2} \E_h^n\|^2_\Omega + \|\mu^{1/2} \H_h^n\|_\Omega^2 \big) \Big|_{\ut}^{\ot} 
=   -\int_{\partial\Omega \times \In} \beta |\n \times \H_h^n|^2 + \n \times \g \cdot \H_h^n  \\
& \qquad \qquad  - \sum_{\K \in \Oh} \big( \tfrac{\eps}{2} \|[\E_h^n(\ut) - \E_h^{n-1}(\ut)] \|^2_\K + \tfrac{\mu}{2} \|\H_h^{n}(\ut) - \H_h^{n-1}(\ut)]\|_\K^2 \big).
\end{align*}
\end{theorem}
\begin{proof}
By Theorem~\ref{thm:norm} and Theorem~\ref{thm:coercivity}, we obtain 
$$
B^n(\vh,\ve;\vh,\ve) \ge\tnorm(\ve,\vh)\tnorm_{\Oh \times \In}^2 \ge c \big( \|\eps^{1/2} \ve\|^2_{\Omega \times \In} + \|\mu^{1/2} \vh\|_{\Omega \times \In}^2 \big)
$$
for all $(\ve,\vh) \in \TT_p(\Oh \times \In)$ with positive constant $c = \min_{\K \in \Oh} C(p,\K,\In) > 0$.
Hence $B^n$ is coercive on $\TT_p(\Oh \times \In)$ and \eqref{eq:stdgt} therefore uniquely solvable.
This shows that the discrete variational problems \eqref{eq:stdgt}  are uniquely solvable for every time step $n \ne 1$. 
Since Method~\ref{meth:stdgt} is a special instance of the Method~\ref{meth:stdg}, the 
consistency and energy dissipation relation follow directly from Theorems~\ref{thm:consistency} and \ref{thm:energy} in Section~\ref{sec:stdg}.
\end{proof}

\section{Discussion} \label{sec:disc}

Before we turn to numerical experiments, let us summarize some of the basic properties of the 
space-time discontinuous Galerkin methods discussed in this paper, in particular, of the method using the Trefftz polynomials.
 
(i) Any space-time discontinuous Galerkin method of the form \eqref{eq:stdg} results in an implicit time-stepping scheme, as long as the variational problems are unqiquely solvable. This is the case for a proper choice of approximation spaces, e.g., for complete tensor product polynomials or the Trefftz polynomials. 

(ii) The discontinuous Galerkin framework provides a high level of flexibility concerning spatial and temporal discretizations, e.g., one can use different, adaptive, and even non-conforming meshes on every time slab. Approximations of arbitrary order with varying polynomial orders are possible; see \cite{Lilienthal14} for some results in this direction. 

(iii) Since the underlying problem is hyperbolic, 
the algebraic system to be solved in every time step will be well-conditioned, as long as the spatial and temporal mesh size are of comparable size; the condition number will however depend moderately on the polynomial degree \cite{Lilienthal14}.

(iv) For a standard space-time discontinuous Galerkin method using tensor product polynomials \cite{Monk14,Lilienthal14}, 
$O(p^4)$ basis functions are associated to every space-time element $K \times \In$.
In contrast to that, only $O(p^3)$ Trefftz polynomials are required.
The use of Trefftz polynomials therefore yields a substantial decrease in the size of the linear systems to 
be assembled and solved in every time step, in particular, when turning to higher order approximations.


(v) The space-time discontinuous Galerkin Trefftz method only involves integrals over element interfaces. 
The evaluation of volume integrals, which is the leading order computational complexity for a traditional discontinuous Galerkin method, can be completely avoided.
The restriction to Trefftz polynomials therefore also substantially reduces the cost of assembling.

(vi) As demonstrated in the proof of Theorem~\ref{thm:rep}, the local Trefftz basis can be constructed systematically and efficiently.
In principle, any polynomial basis for $\PP_p(\K)^3 \times \PP_p(\K)^3$ can be choosen for the initial values on 
the space-time element $\K \times \In$, and can then be extended to a basis for $\TT_p(\K \times \In)$ by 
symbolic solution of Maxwell's equations on this element. 
As demonstrated, additional constraint can be incorporated easily.

(vii) Since the Trefftz method is based on local polynomial approximations 
and only involves standard interface integrals, it can easily be integrated 
in any existing discontinuous Galerkin code. Only the set of basis functions has to be adopted.

\medskip 

Theorem~\ref{thm:stdgt} provides the basic ingredients for a complete error analysis of the space-time discontinuous Galerkin Trefftz  proposed in this paper; see \cite{Lilienthal14,Monk14,MoiPer14} for related results.
A detailed error analysis for Method~\ref{meth:stdgt} would however exceed the scope of 
the current manuscript and will be published elsewhere.
%
%
For illustration of the stability and convergence properties of our method, we will instead 
report on numerical tests in which we observe spectral convergence and optimal convergence orders
with respect to spatial and temporal meshsize.
%
%
With these tests, we also evaluate the dissipation and dispersion behavior of the numerical scheme.

\section{Numerical results} \label{sec:numerics}

For a validation of the discontinuous Galerkin Trefftz method, we present a series of 
numerical tests that illustrate the theoretical statements on stability and dissipativity 
of the method and demonstrate its overall convergence behavior. In Section~\ref{sec:cavity}, we 
show optimal convergence with respect to mesh refinement in space and time, and we verify 
exponential convergence with respect to the polynomial degree for sufficiently smooth 
solutions. In Section~\ref{sec:dispersion}, we analyze the numerical dispersion and dissipation 
behavior. Section~\ref{sec:slits} illustrates the application of our method to the simulation 
of two diffraction experiments.  

\subsection{The general setting} \label{sec:numerics:setting}

All numerical results presented in the following sections correspond to a quasi two-dimensional 
setting with symmetry in one of the coordinate directions,
which allows us to display the results more easily.
We consider a domain of the form $\Omega = \Omega' \times (0,1)$ with $\Omega' \subset \RR^2$ denoting the cross-section for fixed $z \in (0,1)$. We further assume that the material parameters and the fields are independent of $z$ 
and that 
\begin{align} \label{eq:tm}
\E = (0,0,\Eup_3(x,y,t)) 
\qquad \text{and} \qquad 
\H = (\Hup_1(x,y,t),\Hup_2(x,y,t),0).
\end{align}
This setting amounts to three dimensional problems with symmetry and polarization of the electric field in $z$ direction. 
The wave propagation is then governed by Maxwell's equations
\begin{align*}
\eps \dt \E - \curl \H = \zero, 
\qquad 
\mu \dt \H + \curl \E  = \zero \qquad \text{in } \O' \times (0,1) \times (0,T).
\end{align*}
with initial conditions $\E(0)=\E^0$ and $\H(0)=\H^0$ on $\Omega$. We also explicitly incorporate the constraints $\div (\eps \E)=0$ and $\div (\mu \H) = 0$. Note that the latter condition is automatically satisfied due to the form \eqref{eq:tm} of the fields. 
We will consider boundary conditions of the form 
\begin{align*}
\n \times \H &= \zero, \qquad \qquad \! \text{on } \Omega' \times \{0,1\} \times (0,T), \\
\n \times \E + \beta \n \times (\n \times \H) &= \n \times \g, \qquad  \text{on } \partial\Omega' \times (0,1) \times (0,T).
\end{align*}
The first condition reflects the symmetry with respect to the $z$-direction
and the second condition allows to model rather general boundary conditions on the lateral boundaries.
%
%
In all our tests, we utilize the discontinuous Galerkin Trefftz method introduced 
in Section~\ref{sec:stdgt} with local approximation spaces $\widetilde \TT_p^{2D}$ discussed in Section~\ref{sec:lda}. 
%
%

\subsection{Convergence rates} \label{sec:cavity}

Our first test problem models the oscillation of a a cylindrical wave in a rectangular cavity 
$\Omega = \Omega' \times (0,1)$ with $\Omega'=(0,\pi) \times (0,\pi)$ and 
permittivity and permeability set to $\eps=\mu=1$. 
For any $m,n \in \NN$ and $\omega=\sqrt{m^2+n^2}$, the functions
\begin{align*} 
\Hup_1 &= - n \sin \left(m x \right) \cos \left(n y \right) \sin \left( \omega t \right), &
\Hup_2 &= m \cos \left(m x \right) \sin \left(n y \right) \sin \left( \omega t \right), \\
\Eup_3 &= \omega \sin \left(m x \right) \sin \left(n y \right) \cos \left( \omega t \right)
\end{align*} 
solve Maxwell's equations, the divergence constraints $\div (\eps \E)=\div(\mu\H)=0$, the symmetry boundary conditions $\n \times \H=0$ on $\Omega' \times \{0,1\}$, and PEC conditions $\n \times \E = \zero$ on the lateral boundary $\partial\Omega' \times (0,1)$. 
In our simulations, we set $m=n=1$.

Since the solution of the model problem is infinitely differentiable, 
we expect to obtain convergence of arbitrary order when increasing the polynomial degree,
which is what we can observe in practice.
In Figure~\ref{fig:spectral_convergence}, we display the relative errors in the space-time $L^2$-norm obtained by simulation on a uniform $10 \times 10 \times 1$ mesh. 
The time horizon ist set to $T=5 \sqrt{2}$ which amounts to five periods. 
\begin{figure}[!ht]
\centering
\includegraphics[width=0.45\textwidth]{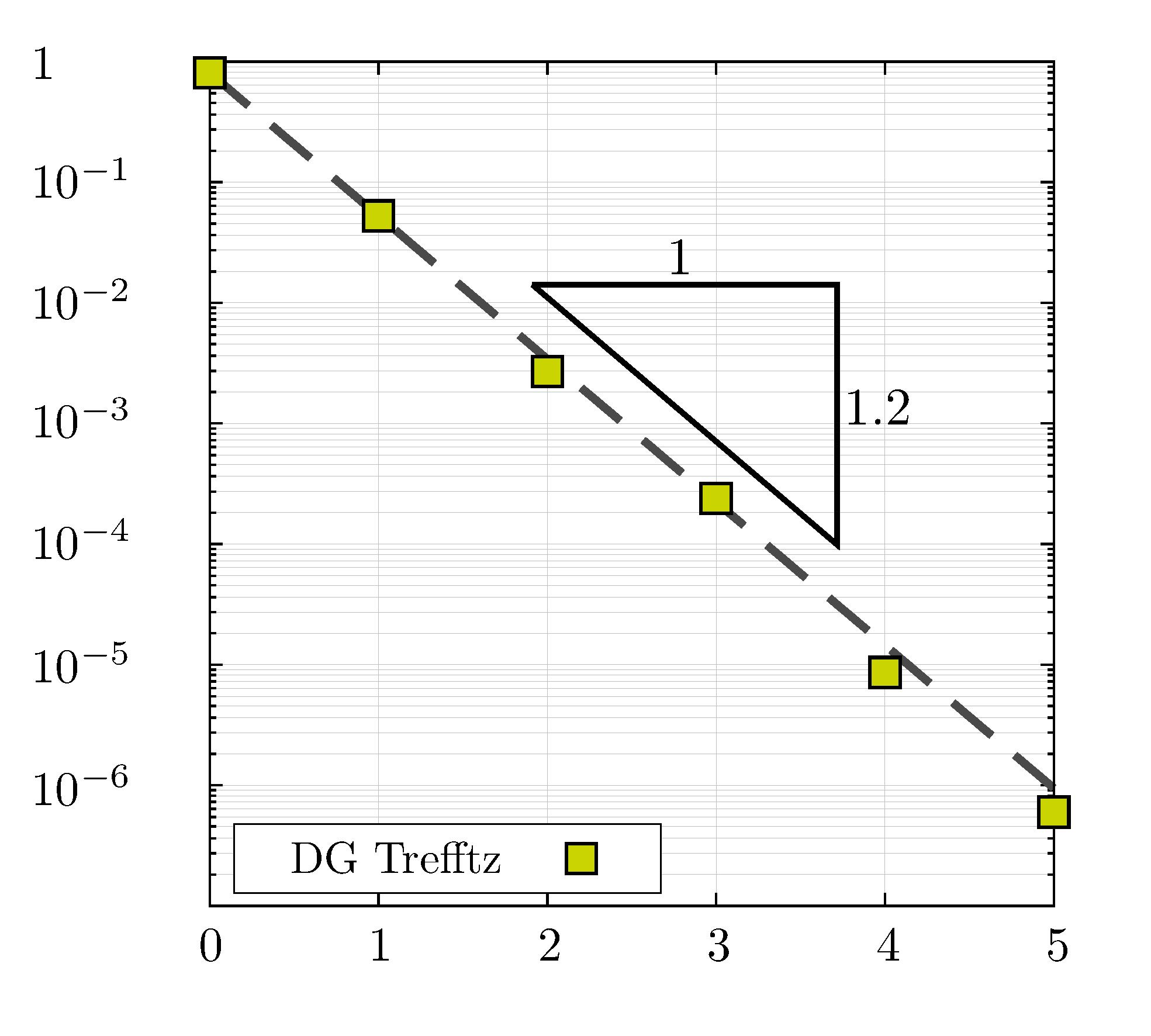} 
  \caption[fig:spectral_convergence]{Relative errors in the space-time $L^2$-norm vs polynomial degree $p$ for the simulation of the cavity resonator problem discussed in Section~\ref{sec:cavity}.} \label{fig:spectral_convergence}
\end{figure}
The logarithmic error plot clearly indicates exponential convergence of the method with respect to the polynomial degree.  

\medskip

In order to assess the convergence orders with respect to the spatial and temporal mesh size, we separately consider refinement in space and time. 
The mesh size in the respective other direction is chosen sufficiently small in order not to affect the convergence. When considering spatial refinement, 
we choose the temporal stepsize corresponding to the finest spatial mesh size, and vice versa. 
\begin{figure}[!ht]
\centering
\includegraphics[width=0.8\textwidth]{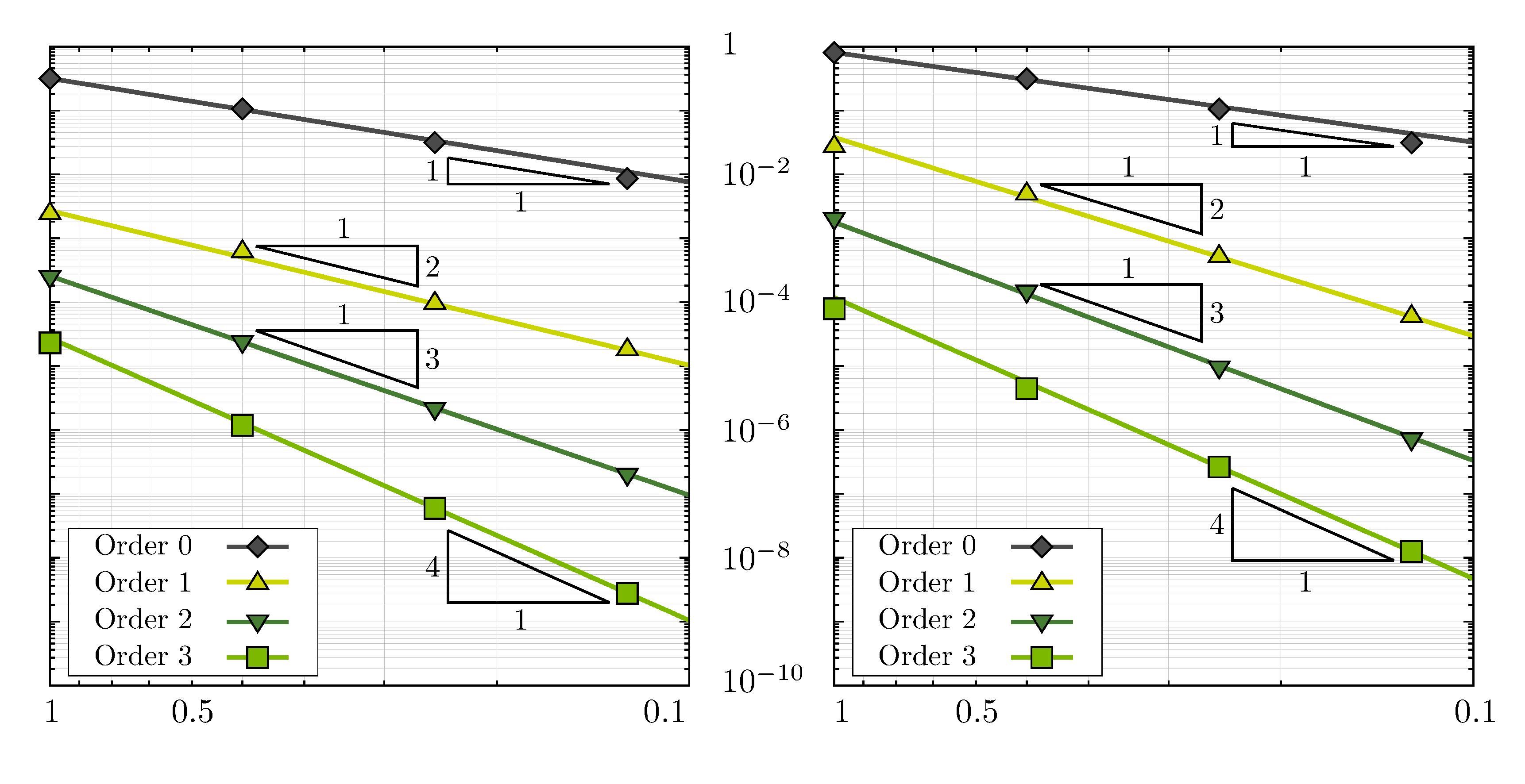}
\caption[fig:convergence_order]{Relative errors in space-time $L^2$-norm for the cavity resonator problem discussed in Section~\ref{sec:cavity}. 
Left: convergence with respect to the spatial refinement; Right: convergence for temporal refinement. For both test series, the coarsest mesh was chosen to consist of $10 \times 10 \times 1$ elements and $50$ time steps. The spatial elements are not refined in $z$ direction. 
The } \label{fig:convergence_order}
\end{figure}
Figure~\ref{fig:convergence_order} displays the the convergence histories for spatial and temporal refinement obtained for different polynomial degrees $p$. 
We observe convergence of order $p+1$ with respect to both, the spatial and temporal mesh size. 
These are the optimal orders concerning the approximation properties of the corresponding full polynomial spaces. 
This optimal behavior or Trefftz methods has already been observed for problems in dimension \cite{Kretzschmar14,MoiPer14}.
In particular, we do not encounter the order reduction of standard discontinuous Galerkin methods reported by \cite{Cohen2006} in our numerical tests. 

\subsection{Numerical dissipation and dispersion} \label{sec:dispersion}

As proved in Sections~\ref{sec:stdg} and \ref{sec:stdgt}, the discontinuous Galerkin methods investigated in this paper 
are slightly dissipative in nature. With the following tests, we aim to quantify the amount and the effect of the numerical dissipation. In addition, we want to evaluate the numerical dispersion of the method, 
i.e., the variation of the propagation velocities for signals of different wave length. 

As a test case, we consider the propagation of a plane wave defined by 
\begin{align} \label{eq:planewave}
\Hup_1=0, \qquad \Hup_2 = \Eup_3 = \psi(x-t)
\end{align}  
with $\psi$ denoting some given function. 
%
For the numerical test below, we set 
\begin{align} \label{eq:theta}
\psi(x)=\Theta \left(\cos (\pi (x-x_0)/10) \right) \cdot \Theta \left(\cos (\pi (x-x_1)/10) \right),
\end{align}
where $\Theta$ is the Heavyside function.  
The solution thus corresponds to a rectangular pulse and is time periodic with period $T_{per}=10$.
As computational domain, we choose $\Omega=\Omega' \times (0,1)$ with $\Omega'=(-10,10) \times (0,3)$.
The electromagnetic fields defined by \eqref{eq:planewave}--\eqref{eq:theta} solve Maxwell's equations 
with $\eps=\mu=1$, the divergence free conditions $\div (\eps \E)=\div(\mu\H)=0$, the symmetry conditions 
$\n \times \H=0$ on $(-10,10) \times (0,3) \times \{0,1\}$ and $(-10,10) \times \{0,3\} \times (0,1)$, 
and periodic boundary conditions at the lateral boundary $\{0,L\} \times (0,H) \times (0,1)$.
%

In Figure~\ref{fig:boxpropagation}, we depict snapshots of the electric field density $\Eup_3$ for one period of 
time.
\begin{figure}[!ht]
 \centering
 \includegraphics[width=0.9\textwidth]{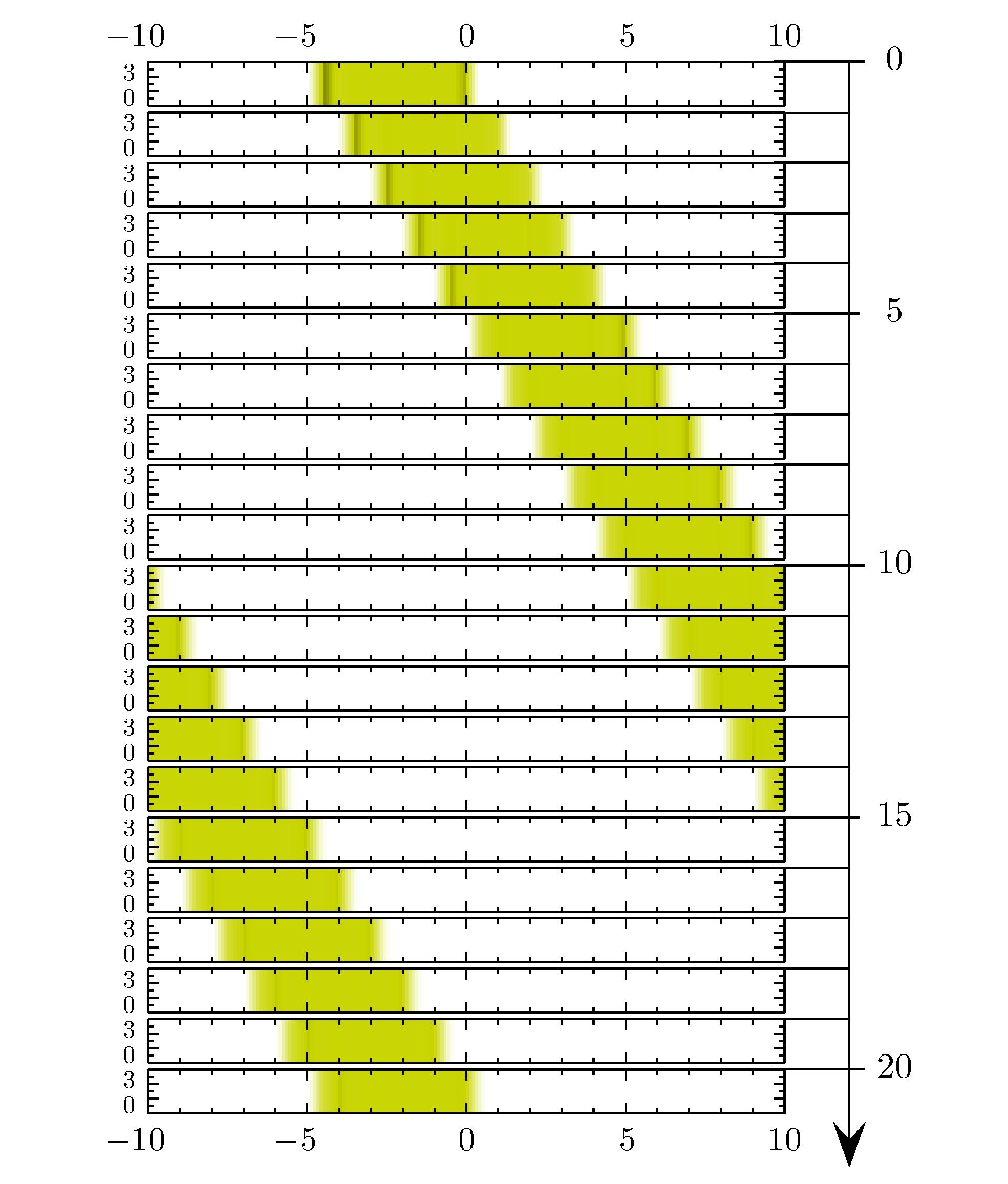}
  \caption[fig:box-prop]{Simulated electric field density $\Eup_3$ of a plane wave \eqref{eq:planewath} propagating from left to right through homogeneous domain with periodic boundary conditions. The results are obtained with order $p=3$ on a uniform mesh with $ 20 \times 3 \times 1$ elements and a temporal stepsize of $\triangle t=1$. 
  The slight overshoots observed at the initial time $t=0$ are a manifestation of the Gibbs phenomenon. 
  Due to numerical dissipation, the overshoots gradually vanish during the propagation.
  After one period, the signal reaches its initial position.} \label{fig:boxpropagation}
\end{figure}
%
Slight overshoots due to the Gibbs phenomenon can be observed at time $t=0$. 
These gradually disappear within a few time steps due to the dissipative behavior of 
the numerical method which damps high oscillations. 
After one period of time, the signal reaches its initial position, which indicates, that the 
wave is propagating at the correct speed. 

\medskip 

For further evaluation of the dissipation and dispersion behavior, we compare in Figure~\ref{fig:disp-error} the electric field amplitudes $\Eup_3(x,t)$ after one, ten, and one hundred periods to the  field at time $t=0$.  
To get insight into the numerical dissipation mechanism, we also display the modulus of the Fourier transformed signals. 
\begin{figure}[!ht]
 \centering
 \includegraphics[width=1.0\textwidth]{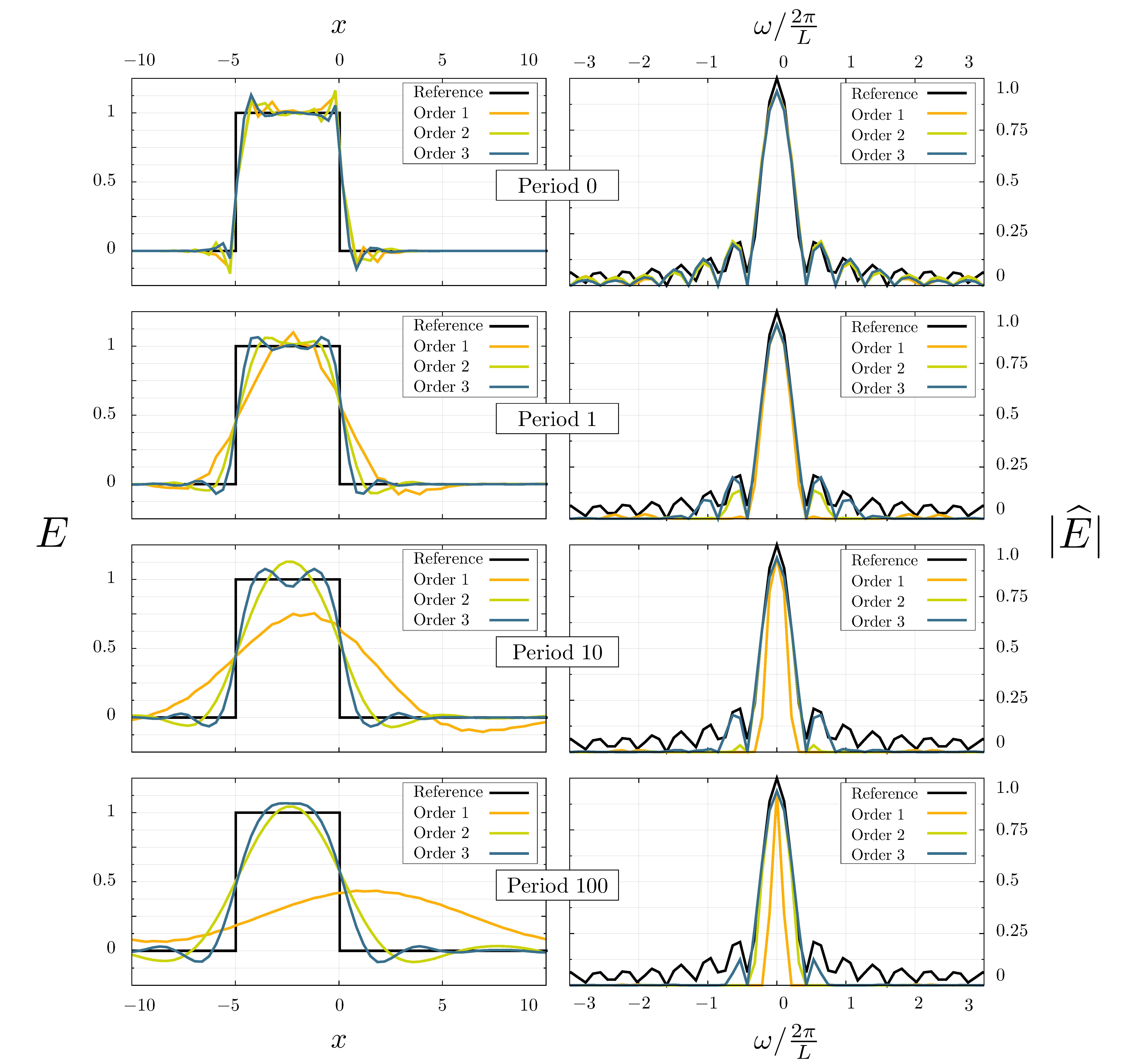}
  \caption[fig:disp-error]{Simulated electric field density $\Eup_3$ corresponding to \eqref{eq:planewave}-
  Left: field amplitude $\E_3(x,1,t)$. Right:  normalized Fourier spectrum $\widehat \E_3(\omega,1,t)$.
  The results for approximation orders $p=1,2,3$ are depicted, respectively, in orange, green, and blue. 
  The exact  solution \eqref{eq:planewave} is displayed in black. The plots are given for time instances $t=0,20,200,2000$ corresponding to $0$, $1$, $10$, and $100$ periods.} \label{fig:disp-error}
\end{figure}
As can be seen from the plots of the Fourier transforms, the dissipation mainly damps the high frequency components and 
therefore acts effectively as a stabilizing mechanism. 
Note that the amount of numerical dissipation decreases with increasing the polynomial degree. 
The plots of the field amplitudes $\Eup_3(x,t)$ reveal that the method with order $p = 1$ shows some significant dispersion while for $p \ge 2$ we can observe hardly any dispersion. 
Even after one hundred periods, the broad band rectangular pulse stays very well located. 
The smearing of the discontinuities can be explained by the numerical dissipation
which can be further reduced by increasing the polynomial approximation order. 

\newpage 

\subsection{Refraction experiments} \label{sec:slits}
To demonstrate the usability of the discontinuous Galerkin Trefftz method in a wider range of applications,
we present simulation results for numerical tests modeling the refraction of a plane wave at slits and materials. 
The initial fields and boundary conditions are chosen as in the plane wave propagation example of the previous section. 
Here, however, we consider the propagation through a material. 

\bigskip 

In the first test case, a rectangular pulse is propagating onto a PEC wall with a small slit. 

\begin{figure}[ht!]
 \centering
 \includegraphics[width=0.7\textwidth]{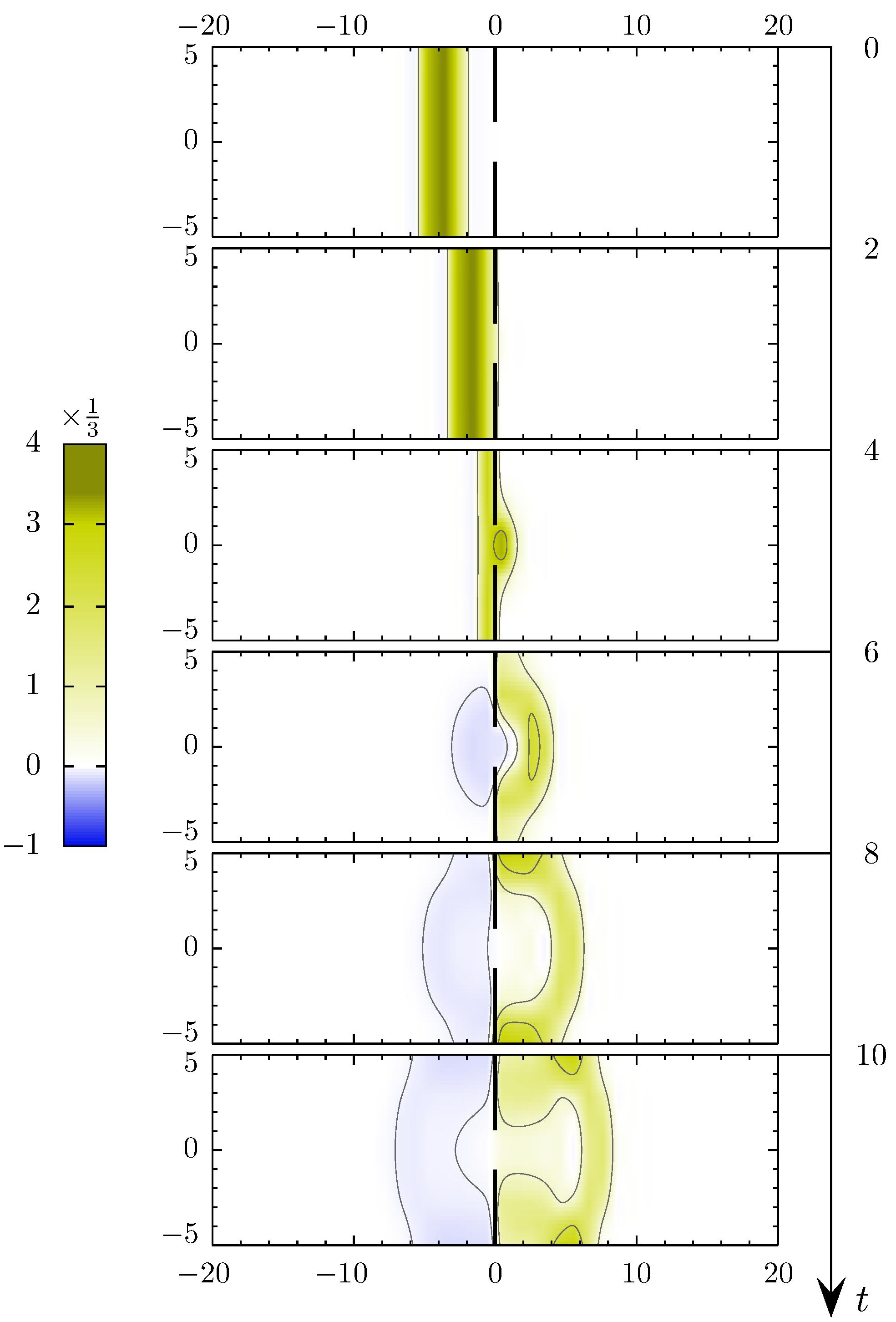}
  \caption[fig:single-slit]{Simulated electric field density $\Eup_3$ of a plane reflected at a PEC wall at $x=0$. 
  Most of the wave is reflected at the wall with change of sign in the electric field. 
  The part propagating through the slit generates an almost cylindrical wave emanating from the center of the slit.} \label{fig:single-slit}
\end{figure}

The simulation results depicted in Figure~\ref{fig:single-slit} show the expected physical behavior: 
Most of the wave is reflected at the wall with a sign change in the electric and a small fraction of 
the field can propagate through the slit and generates an almost cylindrical wave 
emanating from the center of the slit. 
By simulation on larger domains, we expect the method to be applicable for 
obtaining highly accurate diffraction patterns.
%


\begin{figure}[ht!]
 \centering
 \includegraphics[width=0.7\textwidth]{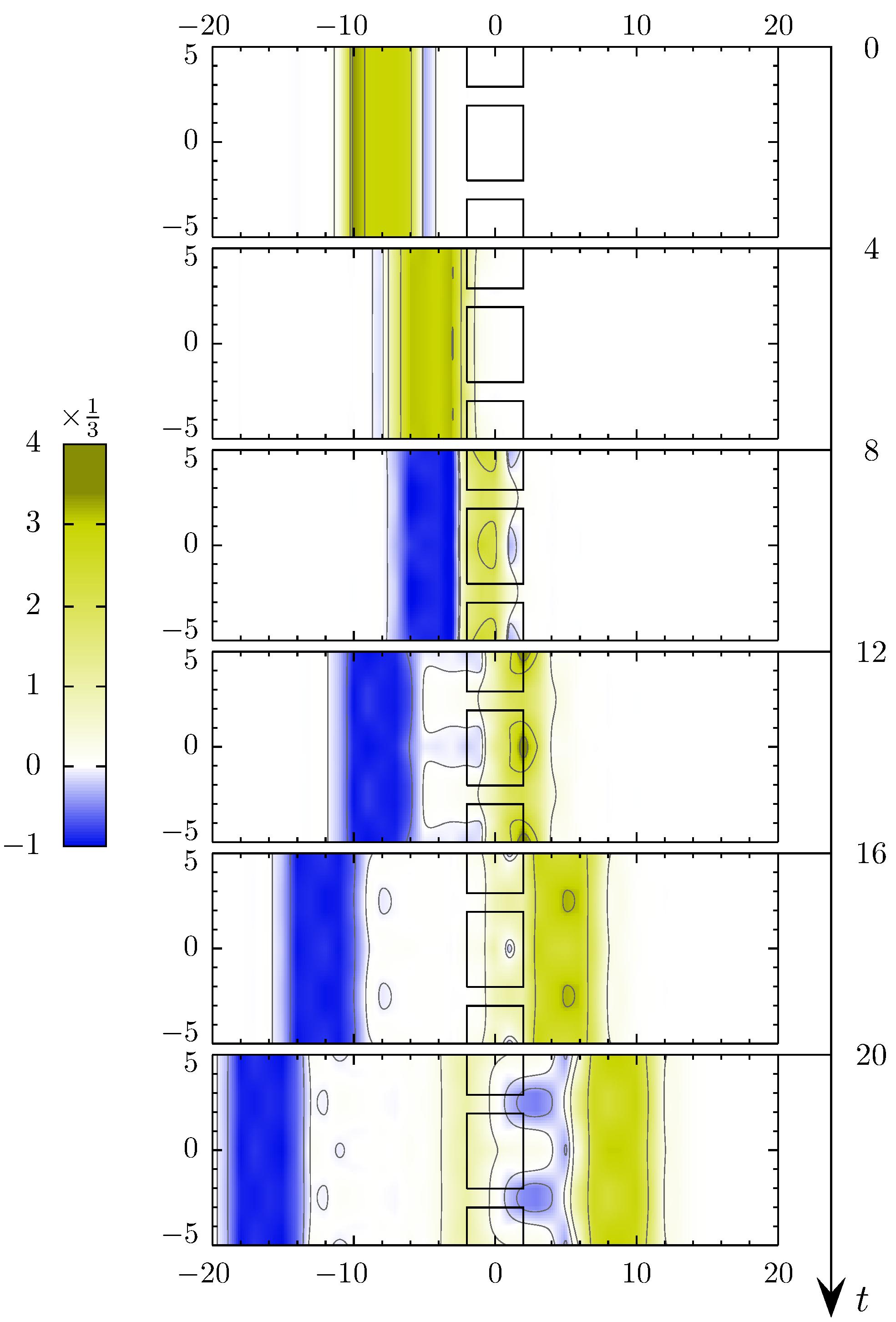}
  \caption[fig:double-slit]{Electric field density $\Eup_3$ of a plane wave propagating through a dielectric double slit with relative permittivity $\eps=4$.  In the free area the material parameters are set to $\eps=\mu=1$.
  The wave is partially reflected at the boundaries of the material and propagates only with lower velocity within 
  the inclusions. } \label{fig:double-slit}
\end{figure}

The second example deals with the propagation of a plane wave through a double slit in a dielectric material of relative permittivity $\eps=4$, while the rest of the domain is covered by a homogeneous material with $\eps=\mu=1$.
Some snapshot of the evolution of the propagating wave are depicted in Figure~\ref{fig:double-slit}.
As expected by physical considerations, the wave first propagates at constant velocity towards the double slit.
Around $t=4$, the wave front impinges on the material and some of the wave is reflected at the material discontinuity. 
%
%
While the wave continues to propagate with constant speed through the areas of the slit, 
the propagation is somewhat slower in the material parts. 
Around $t=12$, a secondary reflection is generated by the wave leaving the material on the right boundary. 
Another one is generated around $t=20$ at the left boundary of the material, a.s.o. 
A typical diffraction pattern is generated behind the slits.

\section{Conclusion}
In this paper, we proposed a discontinuous Galerkin method for electromagnetic wave propagation problems 
based on local approximation with Trefftz polynomials. 
We discussed the explicit construction of a basis for the local spaces of Trefftz polynomials and 
proved some elementary properties of these spaces.  
%
%
%
The resulting discontinuous Galerkin Trefftz method was shown to be well-posed, 
slightly dissipative, and stable with respect to a space-time $L^2$ norm.  
Spectral convergence with respect to the polynomial degree, and optimal convergence rates with respect 
to spatial and temporal meshsize were observed in numerical tests.  
The simulation results indicate that the effect of numerical dissipation becomes negligible for higher 
polynomial approximation orders. For polynomial degree $p \ge 2$, the method showed hardly any 
numerical dispersion.
In comparison to a space-time discontinuous Galerkin method with full polynomial approximation spaces, 
the method based on Trefftz polynomials leads to a substantial reduction in the number of degrees of freedom 
and in the assembling cost, while at the same time, the flexibility of the discontinuous Galerkin framework 
is preserved. For particular applications, the Trefftz method therefore turns out to be a good alternative to 
more standard discontinuous Galerkin approximations.


\appendix 

\section*{Appendix}

\renewcommand\thesection{A}

Let us provide detailed proofs for some of the results stated in the previous sections.
The derivations are more or less straight forward and the arguments are standard in the analysis of discontinuous Galerkin methods. 
The results are presented in detail only for convenience of the reader.

\subsection{Proof of Theorem~\ref{thm:coercivity}} 

We test the variational problem \eqref{eq:stdg} with $\ve=\E_h^n$ and $\vh = \H_h^n$ and apply the following manipulations: 
We first consider the term arising in the right hand side $B^n(\E_h^n,\H_h^n;\E_h^n,\H_h^n)$ of equation \eqref{eq:stdg}. 
By the Leibniz rule for differentiation, we have
\begin{align}
\tag{i} 
\int_\K \big( \eps \dt \E_h^n \cdot \E_h^n + \mu \dt \H_h^n \cdot \H_h^n\big) 
&= \frac{1}{2} \frac{d}{dt} \big(\|\eps^{1/2} \E_h^n\|^2_\K + \|\mu^{1/2} \H_h^n\|^2_\K\big).
\end{align}
The remaining volume terms can be treated via integration-by-parts, e.g., 
\begin{align}
\tag{ii} 
\int_\K \curl \E_h^n \cdot \H_h^n - \curl \H_h^n \cdot \E_h^n 
&= -\int_\dK \n \times \H_h^n \cdot \E_h^n.
\end{align}
Let us now turn to the interface terms:
For interfaces $\f = \dK_1 \cap \dK_2$ between adjacent elements, there holds
\begin{align} \tag{iii}
&\n_1 \times (\H_1 - \{H\}) \cdot \E_1 + \n_2 \times (\H_2 - \{\H\}) \cdot \E_2 
- \n_1 \times (\E_1 - \{E\}) \cdot \H_1  
\\ & \qquad \qquad 
+ \n_2 \times (\E_2 - \{\E\}) \cdot \H_2  
  = \n_1 \times \H_1 \cdot \E_1 + \n_2 \times \H_2 \cdot \E_2\notag
\end{align}
A combination of (i), (ii), (iii), summation over all elements, incorporation of the boundary conditions, and integration over the time interval yields
\begin{align*}
&B^n(\E,\H;\E,\H)
= 
\sum_{\K \in \Oh} \tfrac{\eps}{2} \|\ve(\ot)\|^2_{\K} + \tfrac{\mu}{2} \|\vh(\ot)\|^2_{\K} 
\\ &\qquad \qquad \qquad \qquad 
+ \sum_{\K \in \Oh} \tfrac{\eps}{2} \|\ve(\ut)\|^2_{\K} + \tfrac{\mu}{2} \|\vh(\ut)\|^2_{\K}
+ \sum_{\f \in \Fdh} \int_{\f \times \In} \beta |\n \times \H|^2 .
\end{align*}

\subsection{Proof of Theorem~\ref{thm:energy}}
Using the identity
$$
u v =  \frac{1}{2} u^2 + \frac{1}{2} v^2 - \frac{1}{2} (u-v)^2 
$$
we can express the right hand side of equation \eqref{eq:stdg} as 
\begin{align*}
&R^n(\E_h^{n-1},\H_h^{n-1};\E_h^n,\H_h^n) 
= -\int_{\dO} \n \times \g \cdot \H_h^n + \sum_{\K \in \Oh} \int_\K \eps  \E_h^{n-1} \cdot \E_h^n + \mu \H_h^{n-1} \cdot \H_h^n  \\
&\qquad = -\int_{\dO} \n \times \g \cdot \H_h^n 
        + \frac{1}{2} \sum_{\K \in \Oh} \big(\eps \|\E_h^{n-1}|_\K^2 + \mu \|\H_h^{n-1}\|^2_\K) \\
&\qquad \quad + \frac{1}{2} \sum_{\K \in \Oh} \big(\eps \|\E_h^n\|^2 + \mu \|\H_h^n\|^2_\K\big) 
        - \frac{1}{2} \sum_{\K \in \Oh} \big(\eps \|\E_h^{n-1} - \E\|^2_\K + \mu \|\H_h^{n-1} - \H_h^n\|^2_{\K} \big).
\end{align*}
Theorem~\ref{thm:energy} now follows from the equation $B^n(\E_h^n,\H_h^n;\E_h^n,\H_h^n) = R^n(\E_h^{n-1},\H_h^{n-1};\E_h^n,\H_h^n)$, the representation of $B^n(\E,\H;\E,\H)$ given in the previous proof,  and the above expression for $R(\E_h^{n-1},\H_h^{n-1};\E_h^n,\H_h^n)$  by a slight rearrangement of terms.

\subsection{Proof of identity \eqref{eq:identity}}
By changing the order of summation, one readily obtains 
\begin{align} \tag{iv}
\sum_{\K \in \Oh}  \int_\dK g = \sum_{\f  \in\Fih} (g_1+g_2) + \sum_{\f \in \Fdh} g
\end{align}
for any piecewise smooth function $g$. Here $g_1$, $g_2$ denote the two values of $g$ on the interface $\f = \dK_1 \cap \dK_2$ on the two elements $\K_1$ and $\K_2$, respectively. 
Using this together with formulas for the numerical fluxes $\E^*$ and $\H^*$ yields
\begin{align*}
&\sum_{\K \in \Oh} \int_{\dK}  \n \times (\H - \H^*) \cdot \ve - \n \times (\E-\E^*) \cdot \vh 
= - \sum_{\f \in \Fdh} \int_\f \n \times \E \cdot \vh \\
&\qquad  \qquad 
 + \sum_{\f \in \Fih} \int_\f \n_1 \times \H_1 \cdot \E_1+ \n_2 \times \H_2 \cdot \E_2 + \n_1 \times \H_1 \cdot \E_1  \n_2 \times \H_2 \cdot \E_2\\
&\qquad  \qquad \qquad \qquad \qquad \qquad 
 - \n_1 \times (\E_1 - \{\E\}) \cdot \H_1 - \n_2 \times (\E_2 - \{\E\}) \cdot \H_2. 
\end{align*}
Applying the identity (iii) on every interface element $\f$ already yields the result.

\end{document}